\documentclass[11pt]{article}
\usepackage[T1]{fontenc}
\usepackage{lmodern}
\usepackage{amsmath,amssymb,amsthm}
\usepackage[margin=1.2in]{geometry}
\usepackage{microtype}
\usepackage{needspace}
\theoremstyle{plain}
\newtheorem{theorem}{Theorem}[section]
\newtheorem{lemma}[theorem]{Lemma}
\newtheorem{proposition}[theorem]{Proposition}
\newtheorem{corollary}[theorem]{Corollary}
\theoremstyle{definition}
\newtheorem{definition}[theorem]{Definition}
\newtheorem{example}[theorem]{Example}
\newtheorem{question}[theorem]{Question}
\theoremstyle{remark}
\newtheorem{remark}[theorem]{Remark}
\newcommand{\Z}{\mathbb{Z}}
\newcommand{\N}{\mathbb{N}}
\newcommand{\R}{\mathbb{R}}
\usepackage[hidelinks]{hyperref}
\newif\ifjournal
\ifdefined\JournalVariant \journaltrue \else \journalfalse \fi
\hypersetup{
  pdftitle={Compact Quasi-Metric Spaces Are Expand-Contract Plastic},
  pdfauthor={Thomas Duzer}
}
\ifjournal
\hypersetup{pdfkeywords={quasi-metric space, expand-contract plasticity,
    non-contractive map, hereditarily precompact, compactness,
    Yoneda completeness, left K-sequential completeness,
    asymmetric normed space, unit ball}}
\else
\hypersetup{pdfkeywords={quasi-metric space, expand-contract plasticity,
    non-contractive map, hereditarily precompact, compactness,
    Yoneda completeness, left K-sequential completeness,
    asymmetric normed space, unit ball,
    formal verification, Lean}}
\fi
\title{Compact Quasi-Metric Spaces\\
Are Expand-Contract Plastic}
\author{Thomas Duzer\thanks{thomas.duzer@gmail.com.
  ORCID: \href{https://orcid.org/0009-0007-2147-081X}{0009-0007-2147-081X}.}}
\date{}
\begin{document}
\maketitle

\begin{abstract}
On a totally bounded metric space, every non-contractive self-map is
an isometry. For quasi-metric spaces this fails. Here quasi-metrics
are $T_1$: $d(x, y) = 0$ implies $x = y$. Zavarzina gave a
hereditarily precompact counterexample and asked whether rigidity
holds at least for bijections. The author's toll spiral
(arXiv:2610.13973) shows that it does not. We prove that compactness restores it: every
compact quasi-metric space is expand-contract plastic, meaning that
each of its bijective non-contractive self-maps is an isometry. It is
enough that the space be countably compact, or that its conjugate
topology be compact. Bijectivity is needed: on the natural numbers
there is a compact, hereditarily precompact quasi-metric for which
the shift is non-contractive but not an isometry. For arbitrary
non-contractive self-maps, we prove rigidity when the space is
hereditarily precompact and its conjugate topology is sequentially
compact, in particular when it is hereditarily precompact and Yoneda
complete; neither hypothesis can be dropped. Under these hypotheses
the maps are also surjective, as on compact metric spaces. In both theorems the
$T_1$ axiom can be weakened to $R_0$, $d(x, y) = 0$ implies
$d(y, x) = 0$, if hereditary precompactness is taken to mean that
every sequence has a left K-Cauchy subsequence; surjectivity then
holds up to points at distance $0$ in both directions. The weaker $T_0$ axiom
is not enough: there is a $T_0$ space, compact in both topologies and
satisfying all the other hypotheses of both theorems, with a
bijective non-contractive self-map that is not an isometry. Under $R_0$,
the symmetrized distance is compact exactly when both topologies are
sequentially compact, and that case already follows from the
classical theorem applied to the symmetrized distance, since its
proof also works for pseudo-metrics.
Our theorems also cover spaces whose symmetrized distance is not
totally bounded. In addition, plasticity of a compact quasi-metric
space does not imply plasticity of its symmetrization. Finally, the classical theorem shows that, for an asymmetric norm on
a finite-dimensional space, the symmetric unit ball and sphere are
rigid for every non-contractive self-map, and so are the forward
ones when the norm is $T_1$; without $T_1$, the forward ball need not
be. In infinite
dimension, the symmetric ball and sphere are not plastic in general.
For the forward ball and sphere, we pose the quasi-metric form of the
unit-ball question of Cascales, Kadets, Orihuela and Wingler.
\ifjournal\else
The two theorems are machine-checked in Lean, and so are the
counterexamples showing that their hypotheses cannot be dropped,
apart from the steps listed at the end.
\fi
\end{abstract}

\ifjournal\else, 68V20\fi.\ifjournal\else; formal verification; Lean\fi.

\needspace{5\baselineskip}
\section{Introduction}\label{sec:intro}

A classical rigidity theorem of Freudenthal and Hurewicz
\cite{FH1936}, in the totally bounded form of Naimpally,
Piotrowski and Wingler \cite{NPW2006}, says that a self-map of a
totally bounded metric space cannot expand one pair of points
without contracting another: every non-contractive self-map is an
isometry. Rhodes extended the theorem to totally bounded uniform
spaces \cite{Rhodes1956} (see also \cite{BC1960}); we know of no
version for quasi-uniform spaces. Metric spaces in which every
bijective non-contractive self-map is an isometry are called
\emph{expand-contract plastic}. Plasticity has since been studied
extensively, mostly for unit balls of Banach spaces and for metric
spaces and groups (see \cite{AKZ2019, BMZ2026} and the references
there). Quasi-metric spaces have only recently entered the subject:
Agyingi \cite{Agyingi2026} studies nonexpansive bijections of
$T_0$-quasi-metric spaces and their $q$-hyperconvex hulls. We compare
his results with ours after Theorem~\ref{thm:upper} and in
Example~\ref{ex:sym}.

Zavarzina \cite{Zav2023} studied what remains of this rigidity
without symmetry. She found a hereditarily precompact quasi-metric
space with a non-contractive self-map that is not an isometry
\cite[Theorem~2]{Zav2023}. She also proved that rigidity holds when
both the space and its conjugate are hereditarily precompact
\cite[Theorem~3]{Zav2023}. Her paper closes by asking
whether every bijective non-contractive self-map of a precompact
quasi-metric space is an isometry. The answer is no: the toll spiral
\cite{Duzer2026} is a hereditarily precompact quasi-metric space with
a bijective non-contractive self-map that is not an isometry.

Here we show that every bijective non-contractive self-map of a
compact quasi-metric space is an isometry, and we find conditions
under which the same holds for all non-contractive self-maps. Following \cite{Zav2023}, quasi-metrics
are $T_1$ throughout: $d(x,y) = 0$ implies $x = y$. For $T_0$ spaces
the theorems fail (Theorem~\ref{thm:introB}).

\begin{theorem}[Compactness forces plasticity]\label{thm:introA}
Every compact quasi-metric space is expand-contract plastic: each of
its bijective non-contractive self-maps is an isometry. Countable
compactness suffices, as does compactness of the conjugate
topology.
\end{theorem}

\begin{theorem}[All non-contractive maps]\label{thm:intro}
Let $(X, d)$ be a hereditarily precompact quasi-metric space whose
conjugate topology is sequentially compact; for instance, a
hereditarily precompact space in which every left K-Cauchy sequence
has a Yoneda limit. Then every non-contractive self-map of $X$,
bijective or not, is an isometry.
\end{theorem}

Both proofs rest on returns: points of the orbit of $a$ that come
close to $a$, in each direction of the distance. The toll spiral is
built so that returns in one direction always cost a toll. In Theorem~\ref{thm:introA}, both returns come from
compactness and the inverse map. A subsequence of the forward orbit
of $a$ has a \emph{lower limit}: a point $p$ with
$d(p, F^m a) \to 0$. The inverse of a non-contractive bijection is
non-expansive, so $d(F^{-m} p, a) \le d(p, F^m a) \to 0$. A lower
limit $z$ of the points $F^{-m} p$ then satisfies $d(z, a) = 0$, and
the separation axiom gives $z = a$. This gives forward returns:
$d(a, F^m a) \to 0$ along suitable $m$. Since the inverse is
non-expansive, the backward orbit returns too:
$d(F^{-m} a, a) \le d(a, F^m a) \to 0$. The same compactness
argument gives returns in the other direction,
$d(a, F^{-m} a) \to 0$, along a subsequence. Returns in both
directions then show that $F^{-1}$, which is non-expansive, is also
non-contractive, so $F^{-1}$ and $F$ are isometries.

In Theorem~\ref{thm:intro}, $F$ need not be invertible, so there is
no backward orbit. Hereditary precompactness gives the forward
returns, and the other direction comes from an \emph{upper limit} of
the returns: a point $y$ with $d(F^m a, y) \to 0$. An upper limit of
the points $F^{m-1} a$ along the returns is then mapped to $a$, since
$F$ is an isometry, so $F$ is also onto, as on compact metric
spaces. Both proofs use
separation only to rule out \emph{one-way free moves}, that is, pairs
with $d(x,y) = 0 < d(y,x)$. The absence of such moves is the $R_0$
axiom. Every hypothesis is sharp:

\needspace{7\baselineskip}
\begin{theorem}[Sharpness]\label{thm:introB}
\begin{enumerate}
\item[(a)] There is a $T_0$ quasi-pseudo-metric space, compact in
its own topology and in the conjugate topology, in which every
sequence has a left K-Cauchy subsequence and every left K-Cauchy
sequence has a Yoneda limit, carrying a bijective non-contractive
self-map that is not an isometry.
\item[(b)] There is a compact, hereditarily precompact quasi-metric
space carrying a non-contractive self-map that is not an isometry.
\item[(c)] There is a quasi-metric space with compact conjugate
topology carrying a non-contractive self-map that is not an
isometry.
\end{enumerate}
\end{theorem}

Part (a) is Proposition~\ref{prop:vanish}: a toll spiral whose toll
vanishes toward the deep end, completed by a bottom circle and a top
circle. Parts (b) and (c) are Proposition~\ref{prop:shift}: the shift
on the natural numbers, for a suitable distance and for its
conjugate. Compactness cannot be dropped from
Theorem~\ref{thm:introA} either, as the toll spiral itself shows. For
spaces satisfying the $R_0$ axiom, the picture is therefore complete;
in the table, ``compact'' may be read as sequentially compact.
\begin{center}
\begin{tabular}{@{}lll@{}}
compact topology & bijective non-contractive maps & all
non-contractive maps\\
\hline
$\tau_d$ only & isometries & not always\\
$\tau_{d^{-1}}$ only & isometries & not always\\
both (equivalently, $d_s$ compact) & isometries & isometries
\end{tabular}
\end{center}
For arbitrary maps, it is enough to have hereditary precompactness on
one side, together with compactness (Theorem~\ref{thm:intro}) or
hereditary precompactness \cite[Theorem~3]{Zav2023} on the other.

Section~\ref{sec:balls} brings the asymmetric theory to the central
question of the subject. In finite dimension, the classical theorem
makes the symmetric unit ball and sphere of an asymmetric norm rigid
for every non-contractive self-map, and the forward ones as well when
the norm is $T_1$; $T_1$ cannot be dropped for the forward ball. In infinite dimension, the symmetric ball and
sphere can fail to be plastic: Example~\ref{ex:shift} gives a shift
on $\ell_2(\Z)$ that witnesses the failure, for an asymmetric norm
whose symmetrization is the Hilbert norm. No linear map can witness
such a failure on the forward ball, or, when $p$ is $T_1$, on the
forward sphere. The forward objects
carry the quasi-metric form of the unit-ball question of Cascales,
Kadets, Orihuela and Wingler \cite{CKOW2016}, which we pose as
Question~\ref{q:balls}. For norms, see \cite{HL2026} for recent progress
on the ball question and \cite{LZ2026} for recent work on spheres.

Section~\ref{sec:prelim} recalls the definitions, and the two
theorems are proved in Section~\ref{sec:main}.
Section~\ref{sec:smyth} relates our hypotheses to Smyth
completeness: under hereditary precompactness, sequential Smyth
completeness is equivalent to compactness of the symmetrized metric,
and each one-sided hypothesis is strictly weaker. Section~\ref{sec:toll} shows that no $R_0$ space in which left
K-Cauchy sequences have upper limits, and none in which they have
lower limits, contains a toll spiral with totally bounded orbit.
Section~\ref{sec:open} constructs the counterexamples, and
Section~\ref{sec:balls} turns to asymmetric unit
balls. The two theorems and the counterexamples of
Section~\ref{sec:open} are machine-checked in Lean~4 using
Mathlib, to the extent stated at the end.

\section{Preliminaries}\label{sec:prelim}

A \emph{quasi-pseudo-metric} on a set $X$ is a map $d : X \times X
\to [0, \infty)$ with $d(x,x) = 0$ and $d(x,z) \le d(x,y) +
d(y,z)$. Its \emph{conjugate} is $d^{-1}(x,y) = d(y,x)$, and its
\emph{symmetrization} is the pseudo-metric $d_s = d + d^{-1}$. The
open sets of the topology $\tau_d$ are the sets that contain, around
each of their points $x$, some forward ball
$B_d(x, \varepsilon) = \{y : d(x,y) < \varepsilon\}$. Thus
$x_n \to x$ in $\tau_d$ if and only if $d(x, x_n) \to 0$, and
$x_n \to x$ in the conjugate topology $\tau_{d^{-1}}$ if and only if
$d(x_n, x) \to 0$.

\needspace{5\baselineskip}
We use three separation axioms. The space is
\begin{itemize}
\item[$T_0$] if $d(x,y) = d(y,x) = 0$ implies $x = y$;
\item[$R_0$] if $d(x,y) = 0$ implies $d(y,x) = 0$;
\item[$T_1$] if $d(x,y) = 0$ implies $x = y$.
\end{itemize}
They agree with the separation axioms of the same names for the
topology $\tau_d$, since $x$ lies in the closure of $\{y\}$ exactly
when $d(x,y) = 0$. A space is $T_1$ if and only if it is both $T_0$
and $R_0$.
Following \cite{Zav2023}, a \emph{quasi-metric} is a $T_1$
quasi-pseudo-metric, and \emph{separated} means $T_1$. When $d$ is
$T_0$, its symmetrization $d_s$ is a metric.

A \emph{one-way free move} is a pair with $d(x,y) = 0 < d(y,x)$. The
$R_0$ spaces are exactly those without one-way free moves, and a
$T_0$ space has none exactly when it is $T_1$. We call $x$ and $y$
\emph{isomorphic}, written $x \cong y$, when $d(x,y) = d(y,x) = 0$.
The $T_0$ axiom says exactly that isomorphic points are equal, that
is, that the space is \emph{skeletal}.

Following \cite[Definition~3]{Zav2023}, a sequence $(x_n)$ is
\emph{left K-Cauchy} if for every $\varepsilon > 0$ there is $N$ with
$d(x_n, x_m) < \varepsilon$ whenever $m \ge n \ge N$. In the
enriched-categorical reading of quasi-metric spaces, these are the
\emph{forward-Cauchy} sequences of \cite{BvBR1998}. A
quasi-pseudo-metric space is \emph{precompact} if, for every
$\varepsilon > 0$, finitely many forward balls $B_d(x, \varepsilon)$
cover it, and \emph{hereditarily precompact} if every subspace is
precompact. By \cite[Theorem~3]{KMRV1993}, as stated in
\cite{Zav2023}, a quasi-metric space is hereditarily precompact if
and only if every sequence in it has a left K-Cauchy subsequence.

A self-map $F$ is \emph{non-contractive} if $d(Fx, Fy) \ge d(x,y)$
for all $x, y$, \emph{non-expansive} if $d(Fx, Fy) \le d(x,y)$, and
an \emph{isometry} if equality holds. Zavarzina calls a quasi-metric
space \emph{expand-contract plastic} when every non-expansive
bijection of the space onto itself is an isometry
\cite[Definition~1]{Zav2023}. A bijection $F$ is non-contractive if
and only if $F^{-1}$ is non-expansive, and $F$ is an isometry if and
only if $F^{-1}$ is. So plasticity says equally that every bijective
non-contractive self-map is an isometry.

\begin{definition}[Limits]\label{def:limits}
Let $(x_n)$ be a sequence in $X$.
\begin{itemize}
\item A point $p$ is a \emph{lower limit} of $(x_n)$ if
$d(p, x_n) \to 0$, that is, if $x_n \to p$ in $\tau_d$.
\item A point $y$ is an \emph{upper limit} of $(x_n)$ if
$d(x_n, y) \to 0$, that is, if $x_n \to y$ in $\tau_{d^{-1}}$.
\item A point $y$ is a \emph{Yoneda limit} of $(x_n)$ if
$d(y, z) = \lim_n d(x_n, z)$ for every $z \in X$
\cite{BvBR1998, KS2002}.
\end{itemize}
The space is \emph{left K-sequentially complete} if every left
K-Cauchy sequence has a lower limit, and \emph{sequentially Smyth
complete} if every left K-Cauchy sequence converges in $d_s$. It is
\emph{Yoneda complete} if every left K-Cauchy net has a Yoneda limit
\cite{KS2002}, and \emph{sequentially Yoneda complete} if every left
K-Cauchy sequence has one. Yoneda complete spaces are sequentially
Yoneda complete, and only sequences are used below.
\end{definition}

Every Yoneda limit is an upper limit: taking $z = y$ gives
$\lim_n d(x_n, y) = d(y,y) = 0$. In the standard terminology
\cite{Cobzas2013}, a sequence is \emph{right K-Cauchy} if for every
$\varepsilon > 0$ there is $N$ with $d(x_m, x_n) < \varepsilon$
whenever $m \ge n \ge N$, and a space is \emph{right K-sequentially
complete} if its right K-Cauchy sequences converge in its topology.
Since the left K-Cauchy sequences of $d$ are the right K-Cauchy
sequences of $d^{-1}$, every left K-Cauchy sequence of $d$ has an
upper limit exactly when $(X, d^{-1})$ is right K-sequentially
complete.

\needspace{8\baselineskip}
\begin{lemma}[Cluster points of left K-Cauchy sequences]
\label{lem:cluster}
Let $(x_n)$ be a left K-Cauchy sequence and $p, y \in X$.
\begin{enumerate}
\item[(a)] If for every $\varepsilon > 0$ there are infinitely many
$n$ with $d(p, x_n) < \varepsilon$, then $p$ is a lower limit of
$(x_n)$.
\item[(b)] If for every $\varepsilon > 0$ there are infinitely many
$n$ with $d(x_n, y) < \varepsilon$, then $y$ is an upper limit of
$(x_n)$.
\end{enumerate}
In particular a lower or upper limit of a subsequence is one of the
whole sequence.
\end{lemma}

\begin{proof}
Fix $\varepsilon > 0$ and $N$ with $d(x_n, x_m) < \varepsilon$ for
$m \ge n \ge N$. (a) Choose $n \ge N$ with $d(p, x_n) < \varepsilon$.
Then $d(p, x_m) \le d(p, x_n) + d(x_n, x_m) < 2\varepsilon$ for
every $m \ge n$. (b) Given $n \ge N$, choose $m \ge n$ with
$d(x_m, y) < \varepsilon$. Then
$d(x_n, y) \le d(x_n, x_m) + d(x_m, y) < 2\varepsilon$.
\end{proof}

\begin{lemma}[Completeness as compactness]\label{lem:compact}
Let $(X, d)$ be a quasi-pseudo-metric space.
\begin{enumerate}
\item[(a)] The topology $\tau_d$ is sequentially compact if and only
if it is countably compact, and it is both when it is compact.
\item[(b)] If $\tau_d$ is sequentially compact, then $(X, d)$ is
left K-sequentially complete. If every sequence has a left K-Cauchy
subsequence, the converse holds.
\item[(c)] If $\tau_{d^{-1}}$ is sequentially compact, then every
left K-Cauchy sequence of $d$ has an upper limit. If every sequence
has a left K-Cauchy subsequence, the converse holds.
\item[(d)] Let $(X, d)$ satisfy the $R_0$ axiom. Then $\tau_d$ and
$\tau_{d^{-1}}$ are both sequentially compact if and only if the
pseudo-metric $d_s$ is compact.
\end{enumerate}
\end{lemma}

\begin{proof}
(a) Let $\tau_d$ be countably compact, and let $(x_n)$ be a sequence.
The closed sets $\overline{\{x_k : k \ge n\}}$ are non-empty and
decrease, so they have a common point $p$. Every ball
$B_d(p, 1/k)$ then contains $x_n$ for infinitely many $n$, and
choosing $n_1 < n_2 < \cdots$ with $d(p, x_{n_k}) < 1/k$ gives a
subsequence converging to $p$. Conversely, let $\tau_d$ be
sequentially compact, and suppose that a countable open cover
$U_1, U_2, \dots$ has no finite subcover. Choose
$x_n \notin U_1 \cup \dots \cup U_n$. Some subsequence converges to a
point of some $U_k$, so it eventually lies in $U_k$; but
$x_n \notin U_k$ for $n \ge k$. Finally, a compact topology is
countably compact.

(b) and (c). In the sequentially compact topology, a left K-Cauchy
sequence has a convergent subsequence, and Lemma~\ref{lem:cluster}
passes its limit to the whole sequence. For the converses, a left
K-Cauchy subsequence has a lower limit, respectively an upper limit,
and therefore converges in $\tau_d$, respectively in
$\tau_{d^{-1}}$.

(d) Let both topologies be sequentially compact, and let $(x_n)$ be a
sequence. Some subsequence has a lower limit $p$, and a further
subsequence has an upper limit $y$. Along it,
$d(p, y) \le d(p, x_n) + d(x_n, y) \to 0$, so $d(p, y) = 0$, and
$d(y, p) = 0$ by $R_0$. Hence
$d(x_n, p) \le d(x_n, y) + d(y, p) \to 0$, and the subsequence
converges to $p$ in $d_s$. So the pseudo-metric $d_s$ is
sequentially compact, hence compact. The converse is immediate,
since $d \le d_s$ and $d^{-1} \le d_s$.
\end{proof}

\needspace{10\baselineskip}
\section{Two rigidity theorems}\label{sec:main}

Besides compactness, we use three hypotheses on a
quasi-pseudo-metric space $(X, d)$:
\begin{enumerate}
\item[(i)] every sequence in $X$ has a left K-Cauchy subsequence;
\item[(ii)] every left K-Cauchy sequence in $X$ has an upper limit;
\item[(iii)] $X$ has no one-way free move.
\end{enumerate}
For quasi-metric spaces, (i) is hereditary precompactness. Hypothesis
(ii) says that $(X, d^{-1})$ is right K-sequentially complete, and
(iii) is the $R_0$ axiom. When (i) holds, (ii) is equivalent to
sequential compactness of $\tau_{d^{-1}}$, by
Lemma~\ref{lem:compact}(c).

\begin{lemma}[Returns from compactness]\label{lem:creturns}
Let $(X, d)$ satisfy (iii) and let $\tau_d$ be sequentially compact.
Let $F$ be a bijective non-contractive self-map of $X$, and let $a,
b \in X$. Then there are integers $m_i \ge 1$ with
\[
d(a, F^{m_i}a) < 1/i \qquad\text{and}\qquad d(b, F^{m_i}b) < 1/i
\qquad (i = 1, 2, \dots).
\]
\end{lemma}

\begin{proof}
Let $G = F^{-1}$. Applying the non-contractivity of $F$ to the pair
$(Gx, Gy)$ shows that $G$ is non-expansive, and so is every iterate
$G^n$. By sequential compactness, applied twice, there are $\mu_1 <
\mu_2 < \cdots$ such that $(F^{\mu_k}a)_k$ and $(F^{\mu_k}b)_k$ have
lower limits $p$ and $p'$. Applying $G^{\mu_k}$ to the pair $(p,
F^{\mu_k}a)$,
\[
d(G^{\mu_k}p, a) \le d(p, F^{\mu_k}a) \to 0,
\qquad\text{and likewise}\qquad d(G^{\mu_k}p', b) \to 0 .
\]
By sequential compactness again, there is an infinite set $K$ of
indices along which $(G^{\mu_k}p)_k$ and $(G^{\mu_k}p')_k$ have
lower limits $z$ and $z'$. Then $d(z, a) \le d(z, G^{\mu_k}p) +
d(G^{\mu_k}p, a) \to 0$ along $K$, so $d(z, a) = 0$, and $d(a, z) =
0$ by (iii). Hence $d(a, G^{\mu_k}p) \le d(a, z) + d(z, G^{\mu_k}p)
\to 0$ along $K$, and likewise $d(b, G^{\mu_k}p') \to 0$.

Now let $\varepsilon > 0$. Fix $k \in K$ with $d(a, G^{\mu_k}p) <
\varepsilon$ and $d(b, G^{\mu_k}p') < \varepsilon$, and then $l > k$
with $d(p, F^{\mu_l}a) < \varepsilon$ and $d(p', F^{\mu_l}b) <
\varepsilon$. Since $G^{\mu_k}F^{\mu_l} = F^{\mu_l - \mu_k}$,
\[
d(G^{\mu_k}p, F^{\mu_l - \mu_k}a) \le d(p, F^{\mu_l}a) < \varepsilon,
\]
so $d(a, F^m a) < 2\varepsilon$ for $m = \mu_l - \mu_k \ge 1$, and
likewise $d(b, F^m b) < 2\varepsilon$.
\end{proof}

\begin{theorem}[Compactness forces plasticity]\label{thm:lower}
Let $(X, d)$ be a quasi-pseudo-metric space satisfying (iii) whose
topology $\tau_d$ is sequentially compact. Then every bijective
non-contractive self-map of $X$ is an isometry. The same holds if
$\tau_{d^{-1}}$ is sequentially compact instead.
\end{theorem}

\begin{proof}
Let $F$ be bijective and non-contractive, let $G = F^{-1}$, and let
$a, b \in X$. As in the proof of Lemma~\ref{lem:creturns}, every
iterate $G^n$ is non-expansive.

\emph{Returns of the backward orbit.} Let $(m_i)$ be as in
Lemma~\ref{lem:creturns}. Applying $G^{m_i}$ to the pair $(a,
F^{m_i}a)$,
\[
d(G^{m_i}a, a) \le d(a, F^{m_i}a) < 1/i,
\qquad\text{and likewise}\qquad d(G^{m_i}b, b) < 1/i .
\]

\emph{The other direction for $a$.} By sequential compactness, some
subsequence $(G^{m_{i_k}}a)_k$ has a lower limit $p$. Then
\[
d(p, a) \le d(p, G^{m_{i_k}}a) + d(G^{m_{i_k}}a, a) \to 0,
\]
so $d(p, a) = 0$, and by (iii) $d(a, p) = 0$. Hence $d(a,
G^{m_{i_k}}a) \le d(a, p) + d(p, G^{m_{i_k}}a) \to 0$.

\emph{Conclusion.} Writing $m = m_{i_k} \ge 1$, the triangle
inequality and the non-expansivity of $G^{m-1}$ give
\[
d(a, b) \le d(a, G^m a) + d(G^m a, G^m b) + d(G^m b, b) \le d(a, G^m
a) + d(Ga, Gb) + d(G^m b, b),
\]
and the outer terms tend to $0$ as $k \to \infty$. So $d(a, b) \le
d(Ga, Gb) \le d(a, b)$: the map $G$ is an isometry, and therefore so
is $F$.

For the last statement, note that $F$ is bijective and
non-contractive for $d^{-1}$ as well, and that (iii) is symmetric in
$d$ and $d^{-1}$. Apply what has been proved to $(X, d^{-1})$.
\end{proof}

If $F$ is not bijective, there is no backward orbit, and compactness
alone does not suffice (Proposition~\ref{prop:shift}). The forward
returns must then come from hypothesis (i).

\begin{lemma}[Returns from precompactness]\label{lem:returns}
Let $(X, d)$ satisfy (i), let $F$ be a non-contractive self-map of
$X$, and let $a, b \in X$. Then there are integers $m_i \ge 1$ with
$d(a, F^{m_i}a) < 1/i$ and $d(b, F^{m_i}b) < 1/i$ for $i = 1, 2,
\dots$.
\end{lemma}

\begin{proof}
Every iterate $F^n$ is non-contractive. By (i), applied twice, there
are $n_1 < n_2 < \cdots$ such that $(F^{n_j}a)_j$ and $(F^{n_j}b)_j$
are both left K-Cauchy. For each $i$ choose $j_i$ with $d(F^{n_j}a,
F^{n_l}a) < 1/i$ and $d(F^{n_j}b, F^{n_l}b) < 1/i$ whenever $l > j
\ge j_i$, and put $m_i := n_{j_i + 1} - n_{j_i} \ge 1$. Applying the
non-contractivity of $F^{n_{j_i}}$ to the pair $(a, F^{m_i}a)$,
\[
d(a, F^{m_i}a) \le d(F^{n_{j_i}}a, F^{n_{j_i + 1}}a) < 1/i,
\]
and likewise for $b$.
\end{proof}

\begin{theorem}[Upper limits force isometries]\label{thm:upper}
Let $(X, d)$ be a quasi-pseudo-metric space satisfying (i), (ii),
and (iii). Then every non-contractive self-map of $X$ is an
isometry.
\end{theorem}

\begin{proof}
Let $F$ be non-contractive, $a, b \in X$, and $(m_i)$ as in
Lemma~\ref{lem:returns}.

\emph{A reverse return for $a$, that is, $d(F^m a, a) \to 0$ along
a subsequence.} By (i), some subsequence
$(F^{m_{i_k}}a)_k$ is left K-Cauchy; by (ii) it has an upper limit
$y$. Then
\[
d(a, y) \le d(a, F^{m_{i_k}}a) + d(F^{m_{i_k}}a, y) \to 0,
\]
so $d(a, y) = 0$, and by (iii) $d(y, a) = 0$. Hence $d(F^{m_{i_k}}a,
a) \le d(F^{m_{i_k}}a, y) + d(y, a) \to 0$.

\emph{Conclusion.} Writing $m = m_{i_k} \ge 1$, non-contractivity of
$F^{m-1}$ and the triangle inequality give
\[
d(a, b) \le d(Fa, Fb) \le d(F^m a, F^m b) \le d(F^m a, a) + d(a, b)
+ d(b, F^m b),
\]
and the right side tends to $d(a, b)$ as $k \to \infty$. So $d(Fa,
Fb) = d(a, b)$.
\end{proof}

Under the hypotheses of Theorem~\ref{thm:upper}, $F$ is also onto up
to isomorphism: every point of $X$ is isomorphic to a point of
$F(X)$. Let $a \in X$, and
take $(m_i)$ as in Lemma~\ref{lem:returns} with $b = a$. By (i)
and (ii), some subsequence $(F^{m_{i_k} - 1}a)_k$ has an upper limit
$y$. Since $F$ is an isometry,
$d(F^{m_{i_k}}a, Fy) = d(F^{m_{i_k} - 1}a, y)$, so
\[
d(a, Fy) \le d(a, F^{m_{i_k}}a) + d(F^{m_{i_k} - 1}a, y) \to 0 .
\]
Hence $d(a, Fy) = 0$, and $d(Fy, a) = 0$ by (iii). In a quasi-metric
space this gives $Fy = a$. So, as on compact metric spaces, every
non-contractive self-map of a quasi-metric space satisfying (i)
and (ii) is a surjective isometry, and therefore a bijection. This
applies in particular to Corollary~\ref{cor:yoneda}.

In the symmetric case, the argument of Theorem~\ref{thm:upper} is
the classical recurrence argument for totally bounded metric spaces
\cite{FH1936, NPW2006}. There, a point of the orbit that comes close
to the starting point is close in both directions. In the asymmetric
setting the argument stops, because hereditary precompactness only
gives returns in one direction, $d(a, F^m a) \to 0$. The upper limit
supplies the other direction, $d(F^m a, a) \to 0$. For a bijection,
compactness of either topology suffices, since the backward orbit,
under the non-expansive inverse, returns in both directions. In both
arguments, $R_0$ is the only separation axiom used. We found no
statement of either theorem in the literature on plasticity or on
asymmetric completeness.

What the literature does contain is the symmetrized statement: if
$d_s$ is totally bounded, every non-contractive self-map is an
isometry. The argument is Zavarzina's. The statement is recorded for
quasi-metric spaces in version~1 of \cite{Duzer2026}
(Proposition~7.2), and for
bijections of $T_0$ spaces by Agyingi
\cite[Proposition~2.2]{Agyingi2026}, who also credits
\cite{Zav2023}. Agyingi's study of $q$-hyperconvex hulls, submitted
to arXiv on 7~October 2026, appeared while this paper was in
preparation. He calls a $T_0$-quasi-metric
space \emph{$q$-plastic} when its bijective nonexpansive self-maps
are isometries; this is the present property, read through inverses.
Throughout \cite{Agyingi2026}, compactness means compactness of the
symmetrization. The case of a compact symmetrization is recalled in
Proposition~\ref{prop:smyth}(c) below.

Our theorems go beyond the symmetrized statement. The space of Proposition~\ref{prop:shift} is compact and plastic, and the
space of Example~\ref{ex:chain} satisfies (i)--(iii), yet in both the
symmetrization is not totally bounded. In the $T_0$ setting of
\cite{Agyingi2026}, Proposition~\ref{prop:vanish} shows that
compactness of $\tau_d$ and of $\tau_{d^{-1}}$ does not force
$q$-plasticity. In the other direction, plasticity does not pass to
the symmetrization, even for compact quasi-metric spaces
(Example~\ref{ex:sym}).

The route of Theorem~\ref{thm:lower} was proposed in version~1 of
\cite{Duzer2026} (Section~7), in connection with its Question~7.1,
which asks whether some notion of compactness forces rigidity. The
idea is that the backward returns of a bijection form a left
K-Cauchy sequence, for which completeness would supply a limit.
Theorem~\ref{thm:upper} runs the same returns without bijectivity,
through upper limits. The corollaries below answer the question for
quasi-metric spaces: compactness of either topology forces
plasticity, and hereditary precompactness together with sequential
compactness of the conjugate topology forces every non-contractive
self-map to be an isometry.

\begin{corollary}[Compact quasi-metric spaces are plastic]
\label{cor:compact}
Let $(X, d)$ be a quasi-metric space such that $\tau_d$ or
$\tau_{d^{-1}}$ is countably compact, in particular compact. Then
$X$ is expand-contract plastic. In particular every hereditarily
precompact, left K-sequentially complete quasi-metric space is
plastic.
\end{corollary}

\begin{proof}
A countably compact topology of a quasi-pseudo-metric is
sequentially compact by Lemma~\ref{lem:compact}(a), and (iii)
follows from the $T_1$ axiom; apply Theorem~\ref{thm:lower}. For the
last statement, hereditary precompactness is hypothesis (i)
\cite[Theorem~3]{KMRV1993}, and Lemma~\ref{lem:compact}(b) makes
$\tau_d$ sequentially compact.
\end{proof}

\begin{corollary}[Yoneda completeness forces isometries]
\label{cor:yoneda}
Let $(X, d)$ be a hereditarily precompact quasi-metric space in
which every left K-Cauchy sequence has a Yoneda limit, or whose
conjugate topology is sequentially compact. Then every
non-contractive self-map of $X$, bijective or not, is an isometry.
\end{corollary}

\begin{proof}
Hypothesis (i) of Theorem~\ref{thm:upper} is the sequential
characterization of hereditary precompactness
\cite[Theorem~3]{KMRV1993}; (ii) holds in the first case because Yoneda limits are
upper limits, in the second by Lemma~\ref{lem:compact}(c); (iii)
follows from the $T_1$ axiom.
\end{proof}

\begin{corollary}[Compact symmetrization]\label{cor:sym}
Let $(X, d)$ be a quasi-pseudo-metric space satisfying (iii) whose
topologies $\tau_d$ and $\tau_{d^{-1}}$ are both sequentially
compact; equivalently, whose symmetrization $d_s$ is compact. Then
every non-contractive self-map of $X$ is an isometry.
\end{corollary}

\begin{proof}
The equivalence is Lemma~\ref{lem:compact}(d). A $d_s$-convergent
subsequence is left K-Cauchy, so (i) holds; (ii) holds by
Lemma~\ref{lem:compact}(c); apply Theorem~\ref{thm:upper}.
\end{proof}

Corollary~\ref{cor:sym} is the classical statement for the
pseudo-metric $d_s$, a metric for $T_0$ spaces, recalled in
Proposition~\ref{prop:smyth}(c) and the remark closing its proof.
Without (iii), the equivalence in Corollary~\ref{cor:sym} fails: the
space of Proposition~\ref{prop:vanish} is compact in both topologies,
but its symmetrization is not compact.

\begin{example}[A left K-Cauchy sequence without limits in
Zavarzina's space]\label{ex:zav}
In the space of \cite[Theorem~2]{Zav2023}, $X = \{e^{in} : n \ge
1\}$ with $d(e^{in}, e^{im}) = |e^{in} - e^{im}|$ when $m \ge n$ and
$m \equiv n \pmod 2$, $d(e^{i}, e^{im}) = 2$ for even $m$, and $3$
otherwise. The shift $e^{in} \mapsto e^{i(n+2)}$ is a
non-contractive non-isometry, so Theorem~\ref{thm:upper} predicts a
left K-Cauchy sequence without an upper limit. Indeed, choose even
indices $n_1 < n_2 < \cdots$ such that $e^{in_j}$ converges in the
circle to a point outside $X$. This sequence is left K-Cauchy. But
$d(e^{in_j}, e^{ik}) < 1$ requires $k \ge n_j$ and
$k \equiv n_j \pmod 2$, which fails for every fixed $k$ once
$n_j > k$. The sequence has no lower limit either, since
$d(e^{ik}, e^{in_j}) \ge |e^{ik} - e^{in_j}|$ stays away from $0$.
The shift is not surjective, however, so Theorem~\ref{thm:lower}
predicts nothing here.
\end{example}

\begin{remark}[The two sides]\label{rem:side}
For bijections, the two sides of convergence make no difference:
under (iii), sequential compactness of either topology forces
plasticity. For arbitrary non-contractive maps, neither side suffices
alone (Proposition~\ref{prop:shift}). Theorem~\ref{thm:upper} pairs
compactness on one side with precompactness on the other; applied to
$d^{-1}$, it equally pairs sequential compactness of $\tau_d$ with
hypothesis (i) for $d^{-1}$. Without (iii), compactness does not
force plasticity: the $T_0$ space of Proposition~\ref{prop:vanish} is
compact on both sides, and both theorems fail in it.
\end{remark}

\section{Smyth completeness}\label{sec:smyth}

Under hereditary precompactness, Smyth completeness amounts to
compactness of the symmetrization, the two-sided case of
Corollary~\ref{cor:sym}. That case therefore falls under Zavarzina's
two-sided theorem, recalled below. Each of the one-sided hypotheses
of Section~\ref{sec:main} is strictly weaker.

\begin{theorem}[Zavarzina's two-sided theorem {\cite[Theorem~3]{Zav2023}}]
\label{thm:zav}
Let $(X, d)$ be a quasi-metric space such that $X$ and $(X,
d^{-1})$ are hereditarily precompact. Then every non-contractive
self-map of $X$ is an isometry.
\end{theorem}

\begin{samepage}
\begin{proposition}[Smyth completeness is compactness of the
symmetrization]\label{prop:smyth}
Let $(X, d)$ be $T_0$ and such that every sequence has a left
K-Cauchy subsequence. Then:
\begin{enumerate}
\item[(a)] $X$ is sequentially Smyth complete if and only if $(X,
d_s)$ is compact;
\item[(b)] in that case $(X, d^{-1})$ is hereditarily precompact, so,
when $X$ is a quasi-metric space, it satisfies the hypotheses of
Theorem~\ref{thm:zav};
\item[(c)] in that case every non-contractive self-map of $X$ is an
isometry.
\end{enumerate}
\end{proposition}
\end{samepage}

\begin{proof}
Since $X$ is $T_0$, $d_s$ is a metric.

(a) Suppose that $X$ is sequentially Smyth complete. Every sequence
has a left K-Cauchy subsequence, which then $d_s$-converges, so
$(X, d_s)$ is sequentially compact, hence compact. Conversely, let
$(X, d_s)$ be compact and let $(x_n)$ be left K-Cauchy. Some
subsequence $(x_{n_k})$ converges in $d_s$ to a point $x$. Given
$\varepsilon > 0$, choose $N$ with $d(x_n, x_m) < \varepsilon$ for
$m \ge n \ge N$, and $K$ with $d_s(x_{n_k}, x) < \varepsilon$ for
$k \ge K$ and $n_K \ge N$. For $n \ge n_K$, pick $k \ge K$ with
$n_k \ge n$, and $k' \ge K$ with $n_{k'} \le n$. Then
$d(x_n, x) \le d(x_n, x_{n_k}) + d(x_{n_k}, x) < 2\varepsilon$ and
$d(x, x_n) \le d(x, x_{n_{k'}}) + d(x_{n_{k'}}, x_n) <
2\varepsilon$, so $x_n \to x$ in $d_s$.

(b) Let $A \subseteq X$ and $\varepsilon > 0$. As a subset of a
compact metric space, $A$ has a finite $\varepsilon$-net
$Y \subseteq A$ for $d_s$. Since $d^{-1}(y, x) = d(x, y) \le
d_s(x, y)$, the forward $d^{-1}$-balls of radius $\varepsilon$
centred at the points of $Y$ cover $A$, so $(A, d^{-1})$ is
precompact.

(c) A self-map $F$ that is non-contractive for $d$ is also
non-contractive for $d_s$. A non-contractive self-map of a compact
metric space is an isometry \cite{FH1936}, so $F$ is a
$d_s$-isometry. In the equality
$d(Fx,Fy) + d(Fy,Fx) = d(x,y) + d(y,x)$, each term on the left
dominates its partner on the right, which forces
$d(Fx, Fy) = d(x, y)$. For quasi-metric spaces, (c) also follows
from (b) and Theorem~\ref{thm:zav}. The argument just given is
Zavarzina's, and it uses separation only to make $d_s$ a metric.
Even this is not needed: the recurrence argument of
\cite{FH1936, NPW2006} works unchanged for a totally bounded
pseudo-metric, so (c) holds whenever $(X, d_s)$ is totally bounded,
with no separation axiom.
\end{proof}

Smyth completeness implies the completeness hypotheses of
Corollaries~\ref{cor:compact} and~\ref{cor:yoneda}. Indeed, if
$x_n \to x$ in $d_s$, then $x$ is a lower limit, as the last
statement of Corollary~\ref{cor:compact} requires. Moreover,
$|d(x_n, z) - d(x, z)| \le d_s(x_n, x) \to 0$ for every $z$, so $x$
is also a Yoneda limit. Under hereditary
precompactness, Smyth completeness also yields the two-sided
hypothesis of Theorem~\ref{thm:zav}. Zavarzina's Corollaries~1 and~2 \cite{Zav2023} also yield it, from
one-sided assumptions: a uniform symmetry condition and a bound
$d(x,y) \le C\,d(y,x)$, respectively.

The one-sided hypotheses of Section~\ref{sec:main} do not yield the
two-sided hypothesis,
and neither Theorem~\ref{thm:upper} nor Theorem~\ref{thm:lower}
follows from Proposition~\ref{prop:smyth} or from
Theorem~\ref{thm:zav}. The next three examples show this: the first
two have upper limits, and the third has lower limits. A fourth
example shows that plasticity itself does not pass to the
symmetrization.

\begin{example}[Sequentially Yoneda complete, symmetrization not compact]
\label{ex:chain}
Let $X = \{x_1, x_2, \dots\} \cup \{y\}$ with
\[
d(x_n, x_m) = 2^{-n} - 2^{-m} \ (m \ge n), \quad d(x_m, x_n) = 1 \
(m > n), \quad d(x_n, y) = 2^{-n}, \quad d(y, x_n) = 1 .
\]
This is a quasi-metric. The $T_1$ axiom is clear. For the triangle
inequality, note that every backward move, to a lower index or out
of $y$, costs $1$, while forward moves add up exactly. Every
sequence has a left K-Cauchy subsequence (a constant one, or one
with strictly increasing indices), so $X$ is hereditarily
precompact. A left K-Cauchy sequence is either eventually constant,
or its indices are eventually nondecreasing and unbounded, since a
backward step costs $1$ and $y$ can be followed only by $y$. In the
second case, $y$ is a Yoneda limit of the sequence:
$\lim_n d(x_n, x_m) = 1 = d(y, x_m)$ and
$\lim_n d(x_n, y) = 0 = d(y, y)$.

Yet the points $x_n$ are pairwise at $d_s$-distance at least $1$, so
$(X, d_s)$ is not compact, and $X$ is not sequentially Smyth
complete, by Proposition~\ref{prop:smyth}(a). In fact, the
hypotheses of Theorem~\ref{thm:zav} fail too, not just this
sufficient condition for them. Since $d(x_m, x_n) = 1$ for $m > n$,
the sequence $(x_n)$ has no left K-Cauchy subsequence for $d^{-1}$,
so the conjugate is not hereditarily precompact.
\end{example}

\begin{example}[A non-trivial isometry outside the two-sided theorems]
\label{ex:product}
Let $Z = X \times S^1$, where $X$ is the space of
Example~\ref{ex:chain} and $S^1$ is the circle with its arc metric
$\varrho$ and an irrational rotation $r$, and let
$d_Z((a,b),(a',b')) = \max(d(a,a'), \varrho(b,b'))$. This is a
quasi-metric space, since separation and the triangle inequality
hold coordinatewise. It is hereditarily precompact: refine a
sequence to a subsequence that is left K-Cauchy in the first
coordinate, and then to one that converges in the compact second
coordinate. It is also sequentially Yoneda complete. Indeed, a left
K-Cauchy sequence in $Z$ is left K-Cauchy in each coordinate. By
Example~\ref{ex:chain}, its first coordinate is eventually constant
or has a Yoneda limit, and its second coordinate converges. Both
sequences of distances then converge pointwise, so their maximum
converges to the maximum of the limits, and the pair of limits is a
Yoneda limit in $Z$.

The symmetrization is not compact, since a fibre
$\{x_n\} \times \{b\}$ is $1$-separated. So $Z$ again lies outside
Proposition~\ref{prop:smyth}, and also outside
Theorem~\ref{thm:zav} itself, for the same reason as in
Example~\ref{ex:chain}. On that fibre, the conjugate distance from
any point to any later one is $1$, so the sequence
$((x_n, b))_n$ has no left K-Cauchy subsequence for $d_Z^{-1}$, and
the conjugate is not hereditarily precompact. Finally,
$\mathrm{id} \times r$ is a non-contractive self-map that is an
isometry and not the identity. So Corollary~\ref{cor:yoneda} applies,
outside the two-sided theorems, to a space that carries a non-trivial
isometry: every non-contractive self-map of $Z$ is an isometry.
\end{example}

\begin{example}[Compact, conjugate not hereditarily precompact]
\label{ex:lowerchain}
Move the limit point of Example~\ref{ex:chain} below the chain: let
$X = \{x_1, x_2, \dots\} \cup \{y\}$ with
\[
d(x_n, x_m) = 2^{-n} - 2^{-m} \ (m \ge n), \quad d(x_m, x_n) = 1 \
(m > n), \quad d(y, x_n) = 2^{-n}, \quad d(x_n, y) = 1 .
\]
This is a quasi-metric. The $T_1$ axiom is clear. Every distance is
at most $1$, so the triangle inequality needs checking only along
routes of two moves that each cost less than $1$. These are
$x_n \to x_m \to x_l$ with $n \le m \le l$, where the costs add up
exactly, and $y \to x_n \to x_m$ with $n \le m$, where
$2^{-m} \le 2^{-n} + (2^{-n} - 2^{-m})$. Every sequence has a left
K-Cauchy subsequence (a constant one, or one with strictly
increasing indices), so $X$ is hereditarily precompact. The
topology $\tau_d$ is compact: an open set containing $y$ contains a
ball $B_d(y, \varepsilon)$, hence every $x_n$ with
$2^{-n} < \varepsilon$, that is, all but finitely many points.

So Corollary~\ref{cor:compact} applies, through
Theorem~\ref{thm:lower}, on which it rests. No other result of
Sections~\ref{sec:main} and~\ref{sec:smyth} applies. The sequence
$(x_n)$ is left K-Cauchy but has no upper limit, since
$d(x_n, y) = 1$ and $d(x_n, x_k) = 1$ for $n > k$. So the space is
not sequentially Yoneda complete, and hypothesis (ii) of
Theorem~\ref{thm:upper} fails. Moreover, since $d(x_m, x_n) = 1$ for
$m > n$, the sequence $(x_n)$ has no left K-Cauchy subsequence for
$d^{-1}$, so the conjugate is not hereditarily precompact,
Theorem~\ref{thm:zav} does not apply, and $(X, d_s)$ is not compact.

As in Example~\ref{ex:product}, the product $X \times S^1$ with the
maximum distance is again a hereditarily precompact quasi-metric
space. Its topology is the product topology, hence compact, and its
conjugate is not hereditarily precompact. The map
$\mathrm{id} \times r$ is a bijective isometry of it that is not the
identity.
\end{example}

\begin{example}[Compact and plastic, symmetrization not plastic]
\label{ex:sym}
Let $X = \Z \cup \{o\}$, let $a\colon \Z \to (1, 3/2)$ be strictly
decreasing, for instance $a(m) = 1 + 1/(2(1 + 2^m))$, put $w(n) =
2^{-|n|-1}$, and let
\[
d(m, n) = a(\min(m, n)) \ (m \ne n), \quad d(o, n) = w(n), \quad
d(n, o) = 2 - w(\min(n, 0)) .
\]
This is a quasi-metric. The $T_1$ axiom is clear. For the triangle
inequality, note that the distances out of $o$ are at most $1/2$,
while all other nonzero distances lie in $(1, 2)$. A two-move route
out of $o$ therefore ends with a move costing more than $1$, which
exceeds every $d(o, n)$. A route between two integers costs at least
$3/2$ through $o$ and more than $2$ otherwise, while $a < 3/2$. A
two-move route from an integer to $o$ costs more than $2$. This
gives the triangle inequality. An open set containing $o$ contains a ball $B_d(o, \varepsilon)$,
hence every $n$ with $w(n) < \varepsilon$, that is, all but finitely
many points. So $\tau_d$ is compact, and by
Corollary~\ref{cor:compact} every bijective non-contractive
self-map of $(X, d)$ is an isometry.

Now let $G$ fix $o$ and send $n$ to $n - 1$. For distinct integers
$m$ and $n$, both $d(m, n)$ and $d(n, m)$ equal $a(\min(m, n))$,
which $G$ strictly increases. For the pairs involving $o$, the sum
$d(n, o) + d(o, n)$ equals $2$ for $n \le 0$ and $3/2 + w(n)$ for
$n > 0$, while $\max(d(n, o), d(o, n)) = 2 - w(\min(n, 0))$. Neither
increases with $n$. So $G$ is a bijective non-contractive self-map
of the metric spaces $(X, d_s)$ and $(X, \max(d, d^{-1}))$, and it
is an isometry of neither. Thus neither symmetrization of this
compact, plastic quasi-metric space is plastic. This does not
conflict with plasticity of $d$, since $d(o, n) = w(n)$ decreases
from $n$ to $n - 1$ when $n \le 0$. Nor does it conflict with
Corollary~\ref{cor:sym}: distinct integers are at distance above
$1$ from one another and to $o$, so $\tau_{d^{-1}}$ is not compact
and $X$ is not hereditarily precompact.

In the terms of \cite{Agyingi2026}, where the symmetrization is the
maximum, $(X, d)$ is $q$-plastic and its symmetrization is not
plastic. Conversely, \cite[Proposition~5.1]{Agyingi2026} gives a
$T_0$-quasi-metric that is not $q$-plastic and whose symmetrization
is plastic. So neither property implies the other. Moreover, for
the implication from $d$ to its symmetrization, compactness of
$\tau_d$ together with the $T_1$ axiom does not help. This bears on
\cite[Problem~5.4]{Agyingi2026}, which asks for conditions under
which the two properties are equivalent.
\end{example}

\section{Toll spirals and their completions}\label{sec:toll}

Let $(Y, \rho)$ be a metric space, $\sigma$ a bijective isometry of
$Y$, $y_0 \in Y$, and $c > 0$, and write $\varphi(n) = \sigma^n y_0$.
The \emph{toll spiral} \cite[Definition~3.1]{Duzer2026} is $\Z$ with
\[
d(n, m) = \rho(\varphi(n), \varphi(m)) + \tau(n, m), \qquad \tau(n,
m) = c \text{ if } m < n \text{ and } n \ge 0, \ \text{and } 0
\text{ otherwise}.
\]
We call the integers $n \ge 0$ \emph{recent} and the others
\emph{deep}. A move from $n$ to $m$ is a descent if $m < n$ and an
ascent otherwise, and $n$ is its source. The toll is thus charged on
descents from recent sources. We keep the symbol $\tau$ of
\cite{Duzer2026} for the toll; topologies always carry a subscript,
as in $\tau_d$. When $\varphi$ is injective and
$\varphi(\Z)$ is totally bounded, the toll spiral is a bounded,
hereditarily precompact quasi-metric space, and the shift
$n \mapsto n + 1$ is a bijective non-contractive self-map that is
not an isometry \cite[Theorem~3.10]{Duzer2026}. By
Corollary~\ref{cor:compact}, neither of its topologies is countably
compact, and by Corollary~\ref{cor:yoneda}, it is not sequentially
Yoneda complete. Proposition~\ref{prop:host} shows that the toll spiral cannot be
completed within the $R_0$ class: any $R_0$ space containing it has
a left K-Cauchy sequence with no upper limit, and one with no lower
limit.

\begin{remark}[Self-contained checks]\label{rem:selfcheck}
The facts cited above are easy to check from the definition.

The toll is subadditive, so $d$ satisfies the triangle inequality.
Indeed, suppose that $\tau(n, l) = c$, that is, $l < n$ and
$n \ge 0$. Then either $m < n$, and $\tau(n, m) = c$, or
$m \ge n \ge 0$ and $l < m$, and $\tau(m, l) = c$. Injectivity of
$\varphi$ gives the $T_1$ axiom.

The shift is non-contractive. It moves positions by $\sigma$, so the
$\rho$-part is unchanged, and $\tau(n+1, m+1) \ge \tau(n, m)$, since
$m+1 < n+1$ if and only if $m < n$, while $n+1 \ge 0$ is weaker than
$n \ge 0$. The shift is not an isometry: it maps the pair $(-1, -2)$ to
$(0, -1)$, and $d(0, -1) = d(-1, -2) + c$. Indeed, the two
$\rho$-parts are equal because $\sigma$ is an isometry, while
$\tau(-1, -2) = 0$ and $\tau(0, -1) = c$.

For hereditary precompactness, let $(n_i)$ be a sequence of
integers. If some value occurs infinitely often, $(n_i)$ has a
constant subsequence. Otherwise it has a subsequence that is
strictly increasing, or strictly decreasing with all terms negative.
In either case $\tau(n_i, n_k) = 0$ for $i < k$, so
$d(n_i, n_k) = \rho(\varphi(n_i), \varphi(n_k))$. Since the orbit is totally bounded, a
further subsequence is $\rho$-Cauchy, and that subsequence is left
K-Cauchy.
\end{remark}

\begin{samepage}
\begin{proposition}[No $R_0$ host with limits]\label{prop:host}
Let $\varphi(\Z)$ be totally bounded, and let $(X, d)$ be an $R_0$
quasi-pseudo-metric space containing the toll spiral as a subspace.
Then some left K-Cauchy sequence of $X$ has no upper limit, and some
left K-Cauchy sequence of $X$ has no lower limit. In particular no
sequentially Yoneda complete quasi-metric space, and no countably
compact one, contains the toll spiral.
\end{proposition}
\end{samepage}

\begin{proof}
The isometry $\sigma$ is recurrent on the totally bounded orbit.
Indeed, take a $\rho$-Cauchy subsequence $\sigma^{m_j}y_0$. Then
$\rho(\sigma^{m_l - m_j}y_0, y_0) = \rho(\sigma^{m_l}y_0,
\sigma^{m_j}y_0)$, which is small for $l > j$ large, and
$m_l - m_j$ can be made as large as we like. The same holds for
$\sigma^{-1}$.

\emph{Upper side.} There are integers $0 < n_1 < n_2 < \cdots$
with $\rho(\varphi(n_i), \varphi(0)) \to 0$. Moves between them are
ascents, hence toll-free, and $d(n_i, n_k) \le \rho(\varphi(n_i),
\varphi(0)) + \rho(\varphi(0), \varphi(n_k)) \to 0$, so the sequence
is left K-Cauchy. Suppose it has an upper limit $y$, so that
$d(n_i, y) \to 0$. Since ascents from $0$ are toll-free,
$d(0, y) \le \rho(\varphi(0), \varphi(n_i)) + d(n_i, y) \to 0$. So
$d(0, y) = 0$, and $d(y, 0) = 0$ by $R_0$. Hence
$d(n_i, 0) \le d(n_i, y) + d(y, 0) \to 0$. But every $d(n_i, 0)$ is
at least $c$, since it is a descent to $0$ from a recent source.

\emph{Lower side.} There are deep integers $n_1 > n_2 > \cdots$ with
$\rho(\varphi(n_k), \varphi(0)) \to 0$. Moves between them have deep
sources, hence are toll-free, so $d(n_j, n_l) = \rho(\varphi(n_j),
\varphi(n_l)) \to 0$ and the sequence is left K-Cauchy. Suppose it
has a lower limit $p$, so that $d(p, n_k) \to 0$. Since the ascent
from $n_k$ to $0$ is toll-free, $d(p, 0) \le d(p, n_k) +
\rho(\varphi(n_k), \varphi(0)) \to 0$. So $d(p, 0) = 0$, and
$d(0, p) = 0$ by $R_0$. Hence
$d(0, n_k) \le d(0, p) + d(p, n_k) \to 0$. But every $d(0, n_k)$ is
at least $c$, since it is a descent from the recent source $0$.

The last statement follows because Yoneda limits are upper limits,
and because countable compactness gives lower limits
(Lemma~\ref{lem:compact}).
\end{proof}

Corollary~\ref{cor:compact} already rules out a compact quasi-metric
space that contains a toll spiral and to which its shift extends as
a bijective non-contractive self-map. Proposition~\ref{prop:host}
asks less of the host: the shift need not extend to it, and the host
need not have any particular form (compare
\cite[Proposition~7.4]{Duzer2026}).

By Proposition~\ref{prop:host}, the Yoneda completion of such a toll
spiral \cite{BvBR1998} is no longer separated: it acquires one-way free moves. Concretely, the
completion can be described by left K-Cauchy sequences, with
distance $\hat d((x_n), (y_m)) = \lim_n \lim_m d(x_n, y_m)$, the
points of the spiral being the constant sequences
\cite[Propositions~5.2 and~5.3]{BvBR1998}. By recurrence of
$\sigma^{-1}$, there are deep integers $n_1 > n_2 > \cdots$ with
$\varphi(n_j) \to \varphi(0)$. Moves between them have deep sources,
hence are toll-free, so they form a left K-Cauchy sequence. The
point $\hat y$ it defines satisfies
$\hat d(\hat y, 0) = \lim_j d(n_j, 0) = 0$ and
$\hat d(0, \hat y) = \lim_j d(0, n_j) = c$, since each $d(0, n_j)$
is a descent from the recent source $0$.

\section{The hypotheses cannot be dropped}\label{sec:open}

This section shows that no hypothesis of
Theorems~\ref{thm:lower} and~\ref{thm:upper} can be dropped.

We begin with the $R_0$ axiom. Both theorems use hypothesis (iii)
only to turn a zero distance in one direction into a zero distance
in the other. Proposition~\ref{prop:smyth} already settles the $T_0$
spaces whose symmetrization is compact, whether or not they have
one-way free moves. In the remaining case, that of a $T_0$ space
with one-way free moves and a non-compact symmetrization, both
conclusions can fail at once, for a single bijection.

The proof of Proposition~\ref{prop:host} shows where to look for a
counterexample. With separation, an upper limit of toll-free ascents
can only be the orbit point itself, and the toll makes convergence
to it impossible. Without separation, the limit need not be the orbit
point: the distance from the orbit point to the limit can be $0$
while the distance back is positive. We therefore add such limit
points to a toll spiral, in one layer above it and one below it. We
also replace the toll, which takes only the values $c$ and $0$, by a
fee that tends to $0$ at the deep end, so that the space remains
$T_0$.

\begin{samepage}
\begin{proposition}[The $R_0$ axiom cannot be dropped]
\label{prop:vanish}
There is a $T_0$ quasi-pseudo-metric space $(X, d)$ such that
\begin{enumerate}
\item[(a)] every sequence in $X$ has a left K-Cauchy subsequence;
\item[(b)] every left K-Cauchy sequence in $X$ has a Yoneda limit
and a lower limit, and every left K-Cauchy net has a Yoneda limit;
\item[(c)] the topologies $\tau_d$ and $\tau_{d^{-1}}$ are both
compact;
\item[(d)] $(X, d_s)$ is not compact;
\end{enumerate}
together with a bijective non-contractive self-map of $X$ that is not
an isometry. Hence the conclusions of Theorems~\ref{thm:lower}
and~\ref{thm:upper} both fail without (iii), for one and the same
bijection of a $T_0$ space: without the $R_0$ axiom, hypothesis (i),
Yoneda completeness, and compactness of both topologies together do
not force a bijective non-contractive map to be an isometry.
\end{proposition}
\end{samepage}

\begin{proof}
Fix an irrational $\theta$, the circle $T = \R/\Z$ with its
invariant metric $\rho$, a constant $c > 0$, and a \emph{fee}
$g : \Z \to (0, c]$ that is nondecreasing and not constant, with
$g(n) \to 0$ as $n \to -\infty$ and $g(n) \to c$ as
$n \to +\infty$; for instance $g(n) = c \min(1, 2^n)$. The space has
three layers,
\[
X = B \sqcup \Z \sqcup A, \qquad B = \{\beta_y : y \in T\},
\quad A = \{\alpha_y : y \in T\},
\]
with positions $\pi(\beta_y) = \pi(\alpha_y) = y$ and
$\pi(n) = n\theta$. We order the layers as $B < \Z < A$, and order
$\Z$ by index. Thus $z$ lies \emph{strictly below} $x$ when
$z \in B$ and $x \notin B$, when $z, x \in \Z$ with $z < x$, or when
$z \notin A$ and $x \in A$. Let $h(\beta_y) = 0$, $h(n) = g(n)$,
$h(\alpha_y) = c$, and put
\[
d(x, x') = \rho(\pi x, \pi x') + \tau(x, x'), \qquad \tau(x, x') =
\begin{cases}
h(x) & \text{if } x' \text{ lies strictly below } x,\\
0 & \text{otherwise.}
\end{cases}
\]
So the toll $h(x)$ is charged exactly when $x'$ lies strictly below
$x$. Note that
$h(z) \le h(y)$ whenever $z$ lies strictly below $y$.

\emph{Quasi-pseudo-metric.} Since the $\rho$-part is a
pseudo-metric, it is enough to show that the toll is subadditive.
Suppose that $\tau(x, z) > 0$, so that $\tau(x, z) = h(x)$ and $z$
lies strictly below $x$, and let $y$ be any point. If $y$ also lies
strictly below $x$, then $\tau(x, y) = h(x)$. If not, then $y$ is
at the level of $x$ or above it, because any two points are
comparable, except that the points of $B$ are tied with one another,
and so are the points of $A$. In that case $z$ lies strictly below
$y$ too, and $\tau(y, z) = h(y) \ge h(x)$.

\emph{$T_0$, with one-way free moves.} Two points at distance $0$
from each other in both directions must have the same position, and
neither move between them may carry a toll. Distinct points of the
same layer have distinct positions; for the integers, this is
because $\theta$ is irrational. The pairs of points with the same position in different
layers are $\{\beta_y, \alpha_y\}$, $\{\beta_{n\theta}, n\}$, and
$\{n, \alpha_{n\theta}\}$. In each pair, the move up is free, and
the move down costs $c$, $g(n) > 0$, and $c$, respectively. So $X$
is $T_0$. It is not $R_0$, since each such pair is a one-way free
move.

\emph{(a).} Every sequence has a subsequence that stays in one
layer and whose positions $\rho$-converge. If that layer is $B$ or
$A$, no toll is charged, so this subsequence is already left
K-Cauchy. If it is $\Z$, we may also assume that the indices are
constant, strictly increasing, or strictly decreasing. In the first
two cases there is no toll. In the third,
$d(n_j, n_l) = \rho(n_j\theta, n_l\theta) + g(n_j)$ for $l > j$, and
$g(n_j) \to 0$ since $n_j \to -\infty$, so the subsequence is again
left K-Cauchy.

\emph{(b).} Let $(x_k)$ be left K-Cauchy. Since
$\rho(\pi x, \pi x') \le d(x, x')$, the positions form a
$\rho$-Cauchy sequence; let $q$ be its limit. Then
$\rho(\pi x_k, \pi z) \to \rho(q, \pi z)$ for every $z$. So a point
$y$ at position $q$ is a Yoneda limit as soon as
$\tau(x_k, z) \to \tau(y, z)$ for every $z$. There are four cases.

\emph{Case 1: the sequence is eventually constant.} It is its own
Yoneda limit.

\emph{Case 2: infinitely many terms lie in $A$.} Since the
sequence is left K-Cauchy, there is a $k_0$ such that
$d(x_k, x_l) < c$ for all $l > k \ge k_0$. Leaving $A$ costs at
least $c$, so once a term $x_k$ with $k \ge k_0$ lies in $A$, all
later terms do too, and the sequence is eventually in $A$. Then $\tau(x_k, z)$
equals $c$ for $z \notin A$ and $0$ for $z \in A$, which is
$\tau(\alpha_q, z)$. The Yoneda limit is $\alpha_q$.

\emph{Case 3: only finitely many terms lie in $A$, and for some $M$
infinitely many terms are integers $\ge M$.} Choose $k_0$ such
that no $x_k$ with $k \ge k_0$ lies in $A$, and
$d(x_k, x_l) < g(M)$ for all $l > k \ge k_0$. If $x_k = n$ is an
integer $\ge M$ with $k \ge k_0$, then a move from $n$ to a point
strictly below costs at least $g(n) \ge g(M)$, so every later term
is an integer $\ge n$. The sequence is therefore eventually a nondecreasing
sequence of integers $n_k$. Unless we are in Case~1, it is
unbounded. Then
$\tau(x_k, z) = 0$ for $z \in A$. For $z \notin A$, $\tau(x_k, z)$
equals $g(n_k)$ as soon as $n_k$ exceeds $z$ (at once if $z \in B$),
and this tends to $c$. Again the Yoneda limit is $\alpha_q$.

\emph{Case 4: otherwise.} Only finitely many terms lie in $A$, and
for every $M$ only finitely many terms are integers $\ge M$. Then
$h(x_k) \to 0$, so $\tau(x_k, z) \le h(x_k) \to 0 = \tau(\beta_q,
z)$. The Yoneda limit is $\beta_q$.

In every case $d(\beta_q, x_k) = \rho(q, \pi x_k) \to 0$, so
$\beta_q$ is a lower limit, and $X$ is left K-sequentially complete.
The case analysis does not use that the index set is $\N$. Read
``infinitely many'' as ``cofinally many'' and ``finitely many'' as
``not cofinally many''; the analysis then applies to left K-Cauchy
nets. In Case~3, the net is eventually a nondecreasing net of
integers, which is eventually constant if it is bounded. So $X$ is
Yoneda complete.

\emph{(c).} By (a), the lower limits of (b), and
Lemma~\ref{lem:compact}, $\tau_d$ is sequentially compact, and by
the Yoneda limits of (b), so is $\tau_{d^{-1}}$. Compactness itself
can be seen directly. No toll is charged from $B$, so the forward
ball $B_d(\beta_z, \varepsilon)$ is the full preimage
$\pi^{-1}(B_\rho(z, \varepsilon))$. Given a $\tau_d$-open cover,
choose for each $z \in T$ a member containing $\beta_z$, and hence
containing such a preimage. Finitely many of the arcs
$B_\rho(z, \varepsilon_z)$ cover $T$, and the corresponding members
cover $X$. Likewise, no toll is charged on a move into $A$, so
$B_{d^{-1}}(\alpha_z, \varepsilon) = \{x : d(x, \alpha_z) <
\varepsilon\}$ is again a full preimage, and the same argument shows
that $\tau_{d^{-1}}$ is compact.

\emph{(d).} For integers $n > m \ge 0$,
$d_s(n, m) \ge g(n) \ge g(0) > 0$. So the recent integers are
uniformly separated in the metric $d_s$, which is therefore not
compact.

\emph{The map.} Define $F(\beta_y) = \beta_{y + \theta}$,
$F(n) = n + 1$, $F(\alpha_y) = \alpha_{y + \theta}$. It is a
bijection of each layer. Since $F$ adds $\theta$ to every
position, it leaves the term $\rho(\pi x, \pi x')$ of every distance
unchanged, and it remains to compare the tolls. $F$ preserves the relation ``strictly below'', and
$h(Fx) \ge h(x)$ because $g$ is nondecreasing. Hence
$\tau(Fx, Fx') \ge \tau(x, x')$, and $F$ is non-contractive. It is
not an isometry:
\[
d(n + 1, n) - d(n, n - 1) = g(n + 1) - g(n) \qquad \text{for every } n,
\]
and this is positive for some $n$, since $g$ is not constant.
\end{proof}

By (a), (c), and the lower limits of (b), the space also answers the last question of
\cite[Question~7.5]{Duzer2026} positively.

\begin{remark}[Where the two proofs stop]
\label{rem:stops}
Take $a = n$, an integer. In the proof of Theorem~\ref{thm:upper},
the forward returns are the integers $n + m_i$ with
$(n + m_i)\theta \to n\theta$, and their upper limit is
$y = \alpha_{n\theta}$. Indeed, $d(n, \alpha_{n\theta}) = 0$, as the
proof shows it must be, but $d(\alpha_{n\theta}, n) = c$. In the
proof of Theorem~\ref{thm:lower}, the backward orbit returns through
the integers $n - m_i$, and their lower limit is
$p = \beta_{n\theta}$. Here $d(\beta_{n\theta}, n) = 0$, but
$d(n, \beta_{n\theta}) = g(n) > 0$. The step of the proof of
Lemma~\ref{lem:creturns} that uses (iii) breaks on the same pair
$(\beta_{n\theta}, n)$, although the conclusion of the lemma still
holds in this space, by Lemma~\ref{lem:returns}. In each case, the
argument breaks down exactly where it applies hypothesis (iii) to a
one-way free move.
\end{remark}

\begin{remark}[The fee must vanish, but must not be switched off]
\label{rem:fee}
In this construction, the fee has to tend to $0$ at the deep end,
but it cannot be $0$ there. Each of the two alternatives leads back
to a known obstruction.

If $\inf g > 0$, property (a) fails. Indeed, in the sequence
$-1, -2, -3, \dots$, the distance from any term to a later one is at
least $\inf g$, because the later term lies strictly below. So no
subsequence is left K-Cauchy.

If the fee is switched off on the deep integers, that is, $g = c$ on
$\{n \ge 0\}$ and $g = 0$ on $\{n < 0\}$, we get the toll of
Section~\ref{sec:toll}, and the space is no longer $T_0$: each deep
integer $n$ is isomorphic to its copy $\beta_{n\theta}$. With the
top layer removed, this is the compact, left K-sequentially complete
space of \cite[Proposition~7.3]{Duzer2026}.

Passing to the skeletal quotient does not repair this. A bijective
non-contractive map descends to the quotient by $\cong$ exactly when
it preserves $\cong$. Its inverse always does, being non-expansive,
but the shift breaks the class $\{-1, \beta_{-\theta}\}$: the images
$0$ and $\beta_0$ satisfy $d(\beta_0, 0) = 0$ and
$d(0, \beta_0) = c$. The quotient is
$B \sqcup \{0, 1, 2, \dots\} \sqcup A$. One can check that it is the
$T_0$ quotient of the sequential Yoneda completion of the toll
spiral, with $\beta_0$ as the point $\hat y$ of
Section~\ref{sec:toll}. On the quotient, one can still send
$\beta_{-\theta}$ to $0$ and let $F$ act as before elsewhere. The
resulting map is non-contractive, injective, and not an isometry,
but it misses $\beta_0$. With a fee that is positive but vanishes at
$-\infty$, no deep integer is merged with its copy, so no quotient
is needed and $F$ remains a bijection.
\end{remark}

Next, we show that Theorem~\ref{thm:lower} needs bijectivity, and
that Theorem~\ref{thm:upper} needs hypotheses (i) and (ii). A single
quasi-metric on the natural numbers, together with its conjugate,
does all three.

\begin{samepage}
\begin{proposition}[Bijectivity and hypotheses (i) and (ii) cannot
be dropped]\label{prop:shift}
Let $g : \N \to (0, \infty)$ be nondecreasing and not constant, where
$\N = \{0, 1, 2, \dots\}$, and put
\[
d(n, m) = g(n)\, 2^{-(m - n)} \ (m > n), \qquad d(n, m) = g(n) \ (m
< n), \qquad d(n, n) = 0 .
\]
Let $S(n) = n + 1$ be the shift.
\begin{enumerate}
\item[(a)] $(\N, d)$ is a quasi-metric space, its topology $\tau_d$
is compact, and every sequence in it has a left K-Cauchy
subsequence.
\item[(b)] $S$ is non-contractive for $d$ and for $d^{-1}$, it is not
surjective, and it is not an isometry.
\item[(c)] $(\N, d^{-1})$ is a quasi-metric space with compact
conjugate topology. It satisfies hypotheses (ii) and (iii), and not
(i).
\end{enumerate}
Hence bijectivity cannot be dropped from Theorem~\ref{thm:lower} or
from Corollary~\ref{cor:compact}, even for hereditarily precompact
spaces; hypothesis (ii) cannot be dropped from
Theorem~\ref{thm:upper}, by $(\N, d)$; and hypothesis (i) cannot be
dropped from it, by $(\N, d^{-1})$.
\end{proposition}
\end{samepage}

\begin{proof}
(a) The $T_1$ axiom is clear, since $g > 0$. We check the
triangle inequality $d(n, l) \le d(n, m) + d(m, l)$ for distinct
$n, m, l$. Every distance from $n$ is at most $g(n)$, and the
distance from $n$ to any smaller number is exactly $g(n)$. If
$m < n$, then $d(n, m) = g(n) \ge d(n, l)$. If $m > n$ and $m > l$,
then $d(m, l) = g(m) \ge g(n) \ge d(n, l)$, since $g$ is
nondecreasing. In the remaining case $n < m < l$, and
$d(n, m) = g(n)\, 2^{-(m - n)} \ge g(n)\, 2^{-(l - n)} = d(n, l)$.
The topology $\tau_d$ is compact, because every open set containing
$0$ contains all but finitely many points. Indeed, such a set
contains a ball $B_d(0, \varepsilon)$, and this ball contains every
$m$ with $g(0)\, 2^{-m} < \varepsilon$.

It remains to find left K-Cauchy subsequences. Every sequence in
$\N$ has a constant subsequence or a strictly increasing one,
$n_1 < n_2 < \cdots$. In the second case, after passing to a
further subsequence, we may assume that
$g(n_j)\, 2^{-(n_{j+1} - n_j)} < 2^{-j}$ for every $j$. Then $d(n_j, n_l) \le g(n_j)\, 2^{-(n_{j+1} - n_j)} <
2^{-j}$ for $l > j$, and the subsequence is left K-Cauchy.

(b) For $m > n$, $d(n + 1, m + 1) = g(n + 1)\, 2^{-(m - n)} \ge
d(n, m)$, and for $m < n$, $d(n + 1, m + 1) = g(n + 1) \ge
d(n, m)$. The same inequalities, read with the arguments exchanged,
give non-contractivity for $d^{-1}$. The shift misses $0$, and
$d(n + 1, n + 2) - d(n, n + 1) = (g(n + 1) - g(n))/2$ is positive
for some $n$.

(c) The conjugate of a quasi-metric is a quasi-metric, so (iii)
holds. The conjugate topology of $d^{-1}$ is $\tau_d$, which is
compact, so (ii) holds for $d^{-1}$ by Lemma~\ref{lem:compact}(a)
and~(c). But $d^{-1}(n, m) = d(m, n) = g(m) \ge g(0)$ for $m > n$,
so the sequence $0, 1, 2, \dots$ has no left K-Cauchy subsequence
for $d^{-1}$.

The last statements follow from (a)--(c). The space $(\N, d)$ is
compact and hereditarily precompact, so by Theorem~\ref{thm:upper}
it cannot satisfy (ii). The space $(\N, d^{-1})$ satisfies (ii) and
(iii).
\end{proof}

\needspace{13\baselineskip}
By Theorem~\ref{thm:lower}, every bijective non-contractive
self-map of $(\N, d)$ is an isometry, so the shift, which is not an
isometry, cannot be surjective: it misses $0$. The remaining hypothesis of
Theorem~\ref{thm:lower}, compactness, cannot be dropped either. The
toll spiral of Section~\ref{sec:toll} is a hereditarily precompact
quasi-metric space with a bijective non-contractive self-map that is
not an isometry. A simpler example, though not precompact, is the
quasi-metric of \cite[Proposition~5.1]{Agyingi2026} on $\Z$, given by
$d(m, n) = 1 + (1 + 2^m)^{-1}$ for $m < n$ and $d(m, n) = 2$ for
$m > n$, with the inverse of the shift. Its nonzero distances lie in
$(1, 2]$, so it is $T_1$, in particular $R_0$. In summary:
\nopagebreak
\begin{center}
\begin{tabular}{@{}lll@{}}
theorem & hypothesis & cannot be dropped, by\\
\hline
Theorem~\ref{thm:lower} & compactness & the toll spiral
\cite{Duzer2026}; \cite[Proposition~5.1]{Agyingi2026}\\
 & (iii) & Proposition~\ref{prop:vanish}\\
 & bijectivity & $(\N, d)$ of Proposition~\ref{prop:shift}\\
Theorem~\ref{thm:upper} & (i) & $(\N, d^{-1})$ of
Proposition~\ref{prop:shift}\\
 & (ii) & $(\N, d)$ of Proposition~\ref{prop:shift}; the space of
\cite[Theorem~2]{Zav2023}\\
 & (iii) & Proposition~\ref{prop:vanish}
\end{tabular}
\end{center}

\section{Asymmetric unit balls}\label{sec:balls}

The central question of the plasticity literature is whether the
unit ball of a Banach space is plastic \cite{CKOW2016}. We have not
found the quasi-metric form of the question in the literature. An \emph{asymmetric norm} on a
real vector space $X$ \cite{Cobzas2013} is a map
$p : X \to [0, \infty)$ such that $p(\alpha x) = \alpha p(x)$ for
$\alpha \ge 0$, $p(x + y) \le p(x) + p(y)$, and $p(x) = p(-x) = 0$
only for $x = 0$. It defines the quasi-pseudo-metric $d_p(x, y) = p(y - x)$,
the conjugate asymmetric norm $\bar p(x) = p(-x)$, and the norm
$p^s = \max(p, \bar p)$. We call $p$ \emph{$T_1$} if $p(x) > 0$ for
every $x \ne 0$, that is, if $d_p$ is a quasi-metric. For $d_p$, the axioms $T_1$ and $R_0$ coincide,
because $d_p(x, y) = 0 = d_p(y, x)$ already forces $x = y$.

We study the forward ball and sphere $B_p = \{p \le 1\}$ and
$S_p = \{p = 1\}$, and the symmetric ball and sphere
$B_{p^s} = B_p \cap (-B_p)$ and $S_{p^s}$, each with the restriction
of $d_p$. The backward ball $-B_p$ adds nothing. Since
$d_{\bar p}(x, y) = d_p(y, x)$, a map is non-contractive for $d_p$
exactly when it is non-contractive for $d_{\bar p}$. Moreover,
$x \mapsto -x$ is an isometry from $(-B_p, d_{\bar p})$ onto
$(B_p, d_p)$, because $d_p(-x, -y) = p(x - y) = d_{\bar p}(x, y)$. So
the non-contractive self-maps of $-B_p$ correspond to those of
$B_p$.

\begin{proposition}[Finite-dimensional asymmetric balls are rigid]
\label{prop:balls}
Let $X$ be finite-dimensional and $p$ an asymmetric norm on $X$. Then
every non-contractive self-map of a subset of $X$ that is compact for
the topology of $X$, bijective or not, is an isometry of $d_p$ onto
that subset. This applies to $B_{p^s}$ and $S_{p^s}$, and, when $p$ is
$T_1$, to $B_p$ and $S_p$.
\end{proposition}

\begin{proof}
The function $q(x) = p(x) + p(-x)$ is a norm on $X$: it is sublinear
and symmetric, and it vanishes only at $0$. So
$d_p(x, y) + d_p(y, x) = q(y - x)$ is a metric, and in finite
dimension it induces the topology of $X$. Let $K \subseteq X$ be
compact, and let $F$ be a self-map of $K$ that is non-contractive for
$d_p$. Then $F$ is non-contractive for this metric, and the classical
theorem \cite{FH1936, NPW2006} makes it an isometry of the metric. An
isometry of a compact metric space into itself is onto, so
$F(K) = K$. In the equality
$d_p(Fx, Fy) + d_p(Fy, Fx) = d_p(x, y) + d_p(y, x)$, each term on the
left is at least the corresponding term on the right, so
$d_p(Fx, Fy) = d_p(x, y)$.

The symmetric ball and sphere are the unit ball and unit sphere of
the norm $p^s$, so they are compact. Suppose now that $p$ is $T_1$,
and fix a norm $\|\cdot\|$ on $X$. Being sublinear, $p$ is convex and
finite, hence continuous. It is positive on the compact unit sphere
of $\|\cdot\|$, so $p \ge m\|\cdot\|$ for some $m > 0$. Hence $B_p$
and $S_p$ are bounded, and, being closed, compact.
\end{proof}

For a norm $p$, the proposition is the classical finite-dimensional
case \cite{FH1936, NPW2006}. When $p$ is $T_1$, the isometry
statement also follows from Theorem~\ref{thm:upper}, and this is the
route that is mechanized. With $\|\cdot\|$ and $m$ as in the proof,
there is also $M$ with $p \le M\|\cdot\|$, so
$m\|y - x\| \le d_p(x, y) \le M\|y - x\|$. On a compact set $K$, every
sequence has a convergent subsequence, which is left K-Cauchy for
$d_p$ by the upper bound; this is hypothesis (i). A left K-Cauchy
sequence for $d_p$ is Cauchy for $\|\cdot\|$ by the lower bound, so it
converges to some $y \in K$, and $d_p(x_n, y) \le M\|y - x_n\| \to 0$;
this is hypothesis (ii). Finally, $d_p$ is $T_1$, hence $R_0$, which
is hypothesis (iii). Without $T_1$ there is no lower bound, so the
mechanized theorem does not apply. Hypotheses (i) and (ii) still hold
on every compact set, by the upper bound alone, but (iii) can fail,
and it does on the balls, as shown at the end of this section. The new material in this section is in the two examples
below, the exclusion of linear counterexamples, and
Question~\ref{q:balls}. The $T_1$ hypothesis cannot be dropped, at
least for the forward ball.

\needspace{7\baselineskip}
\begin{example}[A half-line]\label{ex:halfline}
On $X = \R$, let $p(x) = \max(x, 0)$. This is an asymmetric norm,
but it is not $T_1$. Here $B_p = (-\infty, 1]$ is unbounded,
$d_p(x, y) = (y - x)^+$, and $F(x) = 2x - 1$ is a bijection of $B_p$
with $d_p(Fx, Fy) = 2 d_p(x, y)$. So $F$ is non-contractive and not
an isometry.

The symmetric ball $B_{p^s} = [-1, 1]$ behaves differently. A
non-contractive self-map $F$ of it satisfies $F(y) - F(x) \ge y - x$
for $x < y$. Since $F(1) - F(-1) \ge 2$, this forces
$F(\pm 1) = \pm 1$. Then
$2 = (F(1) - F(y)) + (F(y) - F(x)) + (F(x) - F(-1)) \ge 2$, which
forces $F(y) - F(x) = y - x$ throughout, so $F$ is the identity.
Thus in dimension one, without $T_1$, the forward ball is not
plastic but the symmetric ball is, although $R_0$ fails on both.
\end{example}

In infinite dimension, the symmetric ball and sphere are not
plastic in general, even when the symmetrization is the Hilbert
norm.

\needspace{6\baselineskip}
\begin{example}[A shift on the Hilbert ball]\label{ex:shift}
On $X = \ell_2(\Z)$, let $\delta = 1/\sqrt 2$, let
$A = \{f \in B_X : f_n \le \delta \text{ for all } n \ge 0\}$, and
let $p(x) = \sup_{f \in A}\, \langle f, x\rangle$. Then $p$ is
sublinear, and since $\delta B_X \subseteq A \subseteq B_X$, we have
$\|x\|/\sqrt 2 \le p(x) \le \|x\|$. So $p$ is a $T_1$ asymmetric
norm. Moreover, $A \cup (-A) = B_X$. Indeed, if $f_n > \delta$ for
some $n \ge 0$, then $\sum_{m \ne n} f_m^2 < 1/2$, so every other
coordinate of $f$ has modulus below $\delta$, and $-f \in A$. Hence
$p^s = \|\cdot\|$: the symmetric ball is the Hilbert ball, and the
symmetric sphere is the Hilbert sphere.

Let $T$ be the left shift, $(Tx)_n = x_{n+1}$, which is unitary.
Then $p(Tx) = \sup_{g \in SA}\, \langle g, x\rangle$, where $S$ is
the right shift, and
$SA = \{g \in B_X : g_n \le \delta \text{ for all } n \ge 1\}
\supseteq A$. So $p \circ T \ge p$ everywhere. Hence $T$ is
non-contractive for $d_p$ on all of $X$, and it is a bijection of
$B_{p^s}$ and of $S_{p^s}$. It is not an isometry:
$d_p(0, e_0) = p(e_0) = \delta$, while
$d_p(T0, Te_0) = p(e_{-1}) = 1$. On the sphere,
$x = -e_0/2 + (\sqrt 3/2)\,e_1$ and $y = x + e_0$ are unit vectors
with $d_p(x, y) = \delta$ and $d_p(Tx, Ty) = 1$. So when $p^s$ is
the Hilbert norm, which is complete and strictly convex, neither
$(B_{p^s}, d_p)$ nor $(S_{p^s}, d_p)$ is plastic. The same mechanism
produces the non-plastic ellipsoids of the normed literature
\cite[Example~1.1]{Zav2018}, an example that \cite{Zav2018}
attributes to \cite{CKOW2016}. Here the set is the unit ball
of the Hilbert norm and the distance comes from the asymmetric norm
$p$. For the ellipsoids, the set is the unit ball of another norm and
the distance comes from the Hilbert norm. In both examples, the shift relies on
this mismatch.
\end{example}

No linear map can show that the forward ball fails to be plastic,
even without $T_1$. Let $F$ be a linear non-contractive self-map of
$B_p$, bijective or not. Non-contractivity against $x = 0$, scaled
by homogeneity, gives $p \circ F \ge p$. Since
$F(B_p) \subseteq B_p$ and $p$ is the Minkowski functional of its
own ball, $p \circ F \le p$. So $F$ is a $d_p$-isometry. On the
sphere, if $p$ is $T_1$ and $F$ is linear with $F(S_p) \subseteq S_p$,
the same conclusion follows by normalizing any $x \ne 0$ to
$x/p(x) \in S_p$. In particular, the shift of Example~\ref{ex:shift}, which is
linear, cannot be used on $B_p$ or $S_p$. The map of
Example~\ref{ex:halfline} is affine but not linear, as it must be.

\begin{question}\label{q:balls}
Example~\ref{ex:shift} settles the symmetric ball and sphere, so the
question concerns the forward ones. Let $p$ be a $T_1$ asymmetric
norm on an infinite-dimensional space. Is $(B_p, d_p)$
expand-contract plastic? Is $(S_p, d_p)$? When $p$ is a complete
norm, the ball question is that of \cite{CKOW2016}, which answers it
for strictly convex $p$. A natural place to start is the strictly
convex asymmetric case, in which the forward sphere $S_p$ contains
no segment. By the discussion above, a counterexample cannot be
linear. We also do not know whether $B_p$ fails to be plastic
for every asymmetric norm that is not $T_1$, as it does in
Example~\ref{ex:halfline}. Nor do we know whether the ball and the
sphere must be treated separately, as they are for norms
\cite{LZ2026}.
\end{question}

Neither theorem of Section~\ref{sec:main} applies to the unit balls
in infinite dimension. Hypothesis (i) fails on every
infinite-dimensional ball, and for $T_1$ asymmetric norms so does
compactness of either topology. For an asymmetric norm that is not
$T_1$, hypothesis (iii) fails on both balls: if $p(z) = 0$ with
$z \ne 0$, then $d_p(0, tz) = 0 < d_p(tz, 0)$ for small $t > 0$, and
$tz \in B_{p^s}$. On the symmetric ball, no such
theorem could hold (Example~\ref{ex:shift}).

To see that (i) and compactness fail, suppose first that
$(B_p, d_p)$ satisfies (i). Then the
subset $B_{p^s}$ satisfies (i) for $d_p$ and, by negation, for
$d_{\bar p}$. So every sequence in $B_{p^s}$ has a subsequence that
is Cauchy for $p^s$. Thus $B_{p^s}$ is totally bounded, and $X$ is
finite-dimensional. Next, suppose that $\tau_{d_{\bar p}}$ is
sequentially compact on $B_p$, and let $(x_n)$ be a sequence in
$B_{p^s}$. Then $(x_n)$ and $(-x_n)$ have a common subsequence along
which $p(y - x_n) \to 0$ and $p(y' + x_n) \to 0$ for some
$y, y' \in B_p$. Hence $p(y + y') = 0$, so $y' = -y$ by $T_1$, and
$p^s(x_n - y) \to 0$. So $B_{p^s}$ would be compact for $p^s$, and
$X$ would again be finite-dimensional. The same argument, with lower
limits in place of upper limits, gives the same conclusion for
$\tau_{d_p}$.

\section*{Data availability}

{\emergencystretch=2em
The two rigidity theorems are machine-checked in Lean~4 using
Mathlib, in files archived with this paper at Zenodo
(\href{https://doi.org/10.5281/zenodo.24015052}{doi:\nolinkurl{10.5281/zenodo.24015052}}), for a
real-valued quasi-pseudo-metric and with the hypotheses taken
literally.
The same files accompany the arXiv version
(\href{https://arxiv.org/abs/2610.14301}{arXiv:2610.14301}) as
ancillary files. The file \texttt{LowerLimits.lean} contains
Theorem~\ref{thm:lower}
(\texttt{compact\_forces\_\allowbreak plasticity}, and
\texttt{conj\_compact\_forces\_\allowbreak plasticity} for the
conjugate side), Lemma~\ref{lem:creturns}
(\texttt{forward\_return\_\allowbreak of\_compact}), the recurrence
of the backward orbit (\texttt{back\_return}), the statement that the
inverse map is an isometry (\texttt{inverse\_isometry}), the last
statement of Corollary~\ref{cor:compact} in its quasi-pseudo-metric
form, with hypotheses (i) and (iii) and left K-sequential
completeness
(\texttt{lower\_limits\_force\_\allowbreak bijective\_isometries}),
and the sequential form of the non-trivial direction of
Lemma~\ref{lem:compact}(d)
(\texttt{symm\_seq\_compact}). The file \texttt{UpperLimits.lean}
contains Theorem~\ref{thm:upper}
(\texttt{upper\_limits\_\allowbreak force\_isometries}), the
recurrence of its proof, a reverse return for one point and a
forward return for the other at a common time
(\texttt{mixed\_return}),
Corollary~\ref{cor:yoneda} with the $T_1$ axiom and sequential
Yoneda limits as hypotheses
(\texttt{yoneda\_\allowbreak forces\_isometries}), and the theorem
with (ii) replaced by sequential compactness of the conjugate
topology (\texttt{conj\_seq\_compact\_\allowbreak
forces\_isometries}). In these files sequential compactness is
stated as the existence, for every sequence, of a subsequence with a
lower limit, respectively an upper limit; the subsequence case of
Lemma~\ref{lem:cluster} is checked as
\texttt{LowerLimit.\allowbreak of\_subseq} and
\texttt{UpperLimit.\allowbreak of\_subseq}. The file
\texttt{ComparableCompact.lean} (added in version~2) contains the
route to Proposition~\ref{prop:balls} for $T_1$ norms through
Theorem~\ref{thm:upper}
(\texttt{comparable\_\allowbreak compact\_\allowbreak forces\_\allowbreak isometries}):
a quasi-pseudo-metric two-sidedly comparable, with positive lower
constant, to the metric of a compact metric space satisfies
(i)--(iii), so Theorem~\ref{thm:upper} applies; and the exclusion
fact before Question~\ref{q:balls}
(\texttt{linear\_\allowbreak exclusion}): a linear
non-contractive self-map of the forward unit ball, bijective or
not, preserves $p$, from positive homogeneity and nonnegativity
alone. All of these except the two subsequence lemmas are audited
with \texttt{\#print axioms}, and they depend only on the axioms \texttt{propext}, \texttt{Classical.choice}, and
\texttt{Quot.sound}. Several steps are not mechanized: the equivalence of hypothesis
(i) with hereditary precompactness \cite[Theorem~3]{KMRV1993}, which
is cited; the cluster-point form of Lemma~\ref{lem:cluster};
Lemma~\ref{lem:compact}(a), on which the passage from countable
compactness in Corollary~\ref{cor:compact} rests; the converse
half of Lemma~\ref{lem:compact}(c) (the converse half of (b) is the
first step of
\texttt{lower\_\allowbreak limits\_\allowbreak force\_\allowbreak bijective\_\allowbreak isometries}); and the deduction of
Corollary~\ref{cor:sym} from \texttt{symm\_seq\_compact}; and the
surjectivity noted after Theorem~\ref{thm:upper}. In
Section~\ref{sec:balls}, the proof of Proposition~\ref{prop:balls}
from the classical theorem, Examples~\ref{ex:halfline}
and~\ref{ex:shift}, the sphere case of the linear exclusion, and the
closing infinite-dimensional argument are not mechanized either, nor
is the two-sided comparability of a $T_1$ asymmetric norm with a
norm in finite dimension (continuity of a sublinear function and its
positive minimum on the sphere), which supplies the hypothesis of the
mechanized theorem above. Example~\ref{ex:zav},
Proposition~\ref{prop:smyth}, Examples~\ref{ex:chain}--\ref{ex:lowerchain}
and Proposition~\ref{prop:host} are proved on paper only, as are
Remarks~\ref{rem:stops} and~\ref{rem:fee}, apart from the one-way
free moves of Remark~\ref{rem:stops}, checked in
\texttt{VanishingToll.lean}, and the finite parts of
Remark~\ref{rem:fee} checked in \texttt{TollExamples.lean}; the facts
checked in Remark~\ref{rem:selfcheck} are mechanized in
\texttt{TollSpiral.lean} of \cite{Duzer2026}, and the exact increase
by $c$ is the theorem \texttt{Pt2.S\_not\_iso} of
\texttt{TollExamples.lean}, whose space has the toll-spiral distance
on its integers. The computation of the point $\hat y$ at the end of
Section~\ref{sec:toll}, the failure of (iii) for asymmetric norms that
are not $T_1$ in Section~\ref{sec:balls}, and the comparison with
\cite{Agyingi2026} at the end of Example~\ref{ex:sym}, the remark in that
example that $\tau_{d^{-1}}$ is not compact, and the $T_1$ check of
the space of \cite[Proposition~5.1]{Agyingi2026} in
Section~\ref{sec:open} are also proved on paper only.\par}

Proposition~\ref{prop:shift} is machine-checked in the file
\texttt{ShiftExample.lean}, for an arbitrary positive nondecreasing
weight and, in the theorems \texttt{counterexample} and
\texttt{conj\_counterexample}, for the weight $g(n) = n + 1$: the
quasi-metric and $T_1$ properties, sequential compactness of
$\tau_d$, hypothesis (i) (\texttt{seqHP}), the non-contractivity,
non-surjectivity, and strict expansion of the shift, and for the
conjugate distance hypothesis (ii) and the failure of (i). Compactness of $\tau_d$ itself, as opposed to
sequential compactness, is not mechanized.

Proposition~\ref{prop:vanish} is machine-checked in the file
\texttt{VanishingToll.lean}, including (a) and the sequential part
of (b), over an
abstract metric space with an isometric bijection and an injective
orbit, and an abstract fee (positive, nondecreasing, bounded by $c$,
tending to $0$ at $-\infty$ and to $c$ at $+\infty$): the
quasi-pseudo-metric and $T_0$ properties, the one-way free moves,
property (a) for a compact position space (\texttt{seqHP}), the
Yoneda limits and the lower-side limits of (b) for a complete one
(\texttt{yoneda\_complete}, \texttt{lower\_complete}), the uniform
separation of the recent integers in the symmetrization, and the
map: bijective, non-contractive, and strictly expanding wherever the
fee increases. The theorem \texttt{counterexample} assembles, for
the circle, an irrational rotation, and the fee $\min(1, 2^n)$, with
$c = 1$, the
quasi-pseudo-metric and $T_0$ properties, property (a), the Yoneda
limits, and the three properties of the map; the free moves, the
lower limits, and the uniform separation are proved at the abstract
level, which covers the instance. Irrationality is used only to make
$n \mapsto n\theta$ injective. Three parts of the proof are not
mechanized in this file: the compactness statements in (c), the step
from uniform separation to non-compactness in (d), and the extension
of (b) from sequences to nets.

The companion file \texttt{TollExamples.lean} provides the circle
instance (\texttt{circle\_instance}). It also checks the finite parts
of the switched-off case of Remark~\ref{rem:fee}. These are the
coincident pair at $-1$, the separation of its images under the
shift, and, on the skeletal quotient, the properties of the modified
map: it is non-contractive and injective, it expands one pair by
exactly $c$, and it misses $\beta_0$ and no other point. The file
also contains the finite parts of a two-charge variant, which this
paper does not use; the theorem \texttt{Pt2.S\_not\_iso}, stated for
every value of the second charge, is the exact increase by $c$ of
Remark~\ref{rem:selfcheck}.

Example~\ref{ex:sym} is machine-checked in the file
\texttt{SymExample.lean}, for an arbitrary strictly decreasing
profile with values in $(1, 3/2)$ and, in the theorem
\texttt{example\_exists}, for the profile displayed: the
quasi-metric and $T_1$ properties, sequential compactness of
$\tau_d$, plasticity, obtained from
\texttt{compact\_forces\_\allowbreak plasticity}, and for the
inverse $n \mapsto n + 1$ of the map $G$ that it is bijective,
non-expansive for the sum and for the maximum symmetrization, and
strictly contracting on the pair $(0, 1)$ for both. As in
\texttt{ShiftExample.lean}, compactness of $\tau_d$ is mechanized
only in its sequential form.

In all seven files, the axioms printed for each theorem are among
the three listed above.
\begin{samepage}
The manifest \texttt{ArtifactHashes.txt} records the SHA-256 hash of
each Lean file and of the three toolchain pins, and its own SHA-256 hash
is:
\nopagebreak
\begin{center}\small
\texttt{f88dd9cb0d16ec19639b2642a63c47066ba6e0eb7bef5d63dd58597a6479076d}
\end{center}
\end{samepage}

\section*{Acknowledgements}

The author developed the mathematics in a draft in French and
used AI to translate it into a first English version. The author
corrected this version wherever it departed from the French and
rewrote it, keeping some of its standard phrasing. The author
then used AI to check the grammar and reviewed the changes.

\needspace{6\baselineskip}

\end{document}